\documentclass[11pt]{article}
\usepackage[british]{babel}
\usepackage{amssymb, amsmath, amsthm, accents}
\usepackage{hyperref,enumerate}
\usepackage[affil-it]{authblk}
\usepackage{graphics, graphicx, tikz}
\usepackage{comment}
\usepackage{color}
\newtheorem{thm}{Theorem}[section]
\newtheorem{conjecture}[thm]{Conjecture}
\newtheorem{cor}[thm]{Corollary}
\newtheorem{lem}[thm]{Lemma}

\newtheorem{cla}[]{Claim}

\newtheorem{case}{Case}

\numberwithin{subcase}{case}
\newcommand{\h}{\mathcal H}
\newcommand{\X}{\mathcal X}
\newcommand{\K}{\mathcal K}
\newcommand{\C}{\mathcal C}
\renewcommand{\P}{\mathcal P}

\title{Partitioning 2-coloured complete $k$-uniform hypergraphs into monochromatic $\ell$-cycles}

\author[1]{Sebasti\'an Bustamante\thanks{Both authors acknowledge support by Millenium Nucleus Information and Coordination in Networks ICM/FIC RC130003.}$^{,}$\thanks{The first author also was supported by CONICYT Doctoral Fellowship 21141116.}$^{,}$}

\author[1]{Maya Stein$^{*,}$\thanks{The second author also received support by Fondecyt Regular grant 1140766 and CMM-Basal AFB 170001.}$^{,}$}

\affil[1]{Department of Mathematical Engineering, University of Chile}

\begin{document}

\maketitle

\begin{abstract}
	We show that for all $\ell, k, n$ with $\ell \leq k/2$ and  $(k-\ell)$ dividing $n$ the following hypergraph-variant of Lehel's conjecture  is true. Every $2$-edge-colouring of the $k$-uniform complete hypergraph $\K_n^{(k)}$ on~$n$ vertices has at most two disjoint monochromatic $\ell$-cycles in different colours that together cover all but at most $4(k-\ell)$ vertices. If $\ell \leq k/3$, then at most two $\ell$-cycles cover all but at most $2(k-\ell)$ vertices. \\ Furthermore, we can cover all vertices with at most $4$ ($3$ if $\ell\leq k/3$) disjoint monochromatic  $\ell$-cycles.
\end{abstract}

\section{Introduction}

Cycle partitioning problems originated in Lehel's Conjecture~\cite{Aye79}, 
which states that every 2-edge-colouring of the complete graph $K_n$ contains two disjoint monochromatic cycles in different colours covering all vertices of~$K_n$, where vertices and single edges  count as (degenerate) cycles.
Lehel's Conjecture was confirmed for large $n$ in~\cite{All08,LRS98}, and for all $n$ by Bessy and Thomass\'e~\cite{BT10}. 

Not much is known for extensions
of this question to $k$-uniform hypergraphs. 
Probably the most flexible concept for cycles in hypergraphs is the notion of {\em  $\ell$-cycles,} which are  $k$-uniform hypergraphs with a cyclic ordering of their vertices and a cyclic ordering of their edges such that every edge contains $k$ consecutive vertices, and consecutive edges intersect in exactly $\ell$ vertices. 
We also consider two edges intersecting in $2\ell$ vertices as a cycle, and vertex sets of size $k - \ell$  as {\em  degenerate cycles}\footnote{We could define any proper subset (of any size) of an edge as a  degenerate edge/cycle, but in our arguments we only need to consider sets of size $k-\ell$ anyway.}. Notice that $1$-cycles and $(k-1)$-cycles often appear in the literature as {\em loose} and {\em tight} cycles, respectively.

It follows from work of Gy\'arf\'as and S\'ark\"ozy~\cite{GS14} that the number of monochromatic 1-cycles needed to partition any 2-edge-coloured $\K_n^{(k)}$ is bounded by a function in $k$.\footnote{The actual result in~\cite{GS14} is on cycle partitioning in $r$-coloured complete hypergraphs (for arbitrary $r\geq 2$), following a recent surge of activity around cycle partitioning in $r$-coloured complete (and other) graphs. For an overview of the area, we recommend the  surveys~\cite{FLM15} and~\cite{Gya16}.} 
The same authors conjectured~\cite{Gya16, GS14} that any 2-edge-coloured $\K_n^{(k)}$ has two disjoint monochromatic 1-paths (an $\ell$-path is obtained from an $\ell$-cycle by deleting one edge),  together covering all but at most $k-2$ vertices, and show this is best possible. This conjecture has recently been confirmed by Lu, Wang and Zhang~\cite{LuWangZhang}.

The $\ell$-cycle partition problem for larger $\ell$ has not been resolved so far.
In~\cite{BHS17},  H\'an and the present authors use hypergraph regularity to prove that all but $o(n)$ of the vertices of $\K_n^{(3)}$ can be covered by two monochromatic vertex-disjoint $2$-cycles of different colours, and an analogous statement for $1$-cycles also holds.

Here we show that for arbitrary $k$, and  $\ell\leq k/2$ the bound on the number of uncovered vertices can be improved to a constant (depending on $k$). More precisely, all but at most
 $5(k-\ell)-1$ vertices can be covered by two disjoint  monochromatic $\ell$-cycles of different colours
 \footnote{For later convenience, we state our results requiring that $k-\ell$ divides $n$ (the order of the host hypergraph). But clearly, our theorem immediately implies that for all $\ell, k, n\in \mathbb N$ with
 $0<\ell \leq k/2$, and any 2-edge-colouring of~$\K_n^{(k)}$  there are two vertex-disjoint monochromatic $\ell$-cycles in different colours together covering all but at most $5(k-\ell)-1$ vertices.
}.

\begin{thm}\label{thm:RB-ell-cycles}
 Let $\ell, k, n\in \mathbb N$ such that
 $0<\ell \leq k/2$ and $k-\ell$ divides $n$. Let any 2-edge-colouring of~$\K_n^{(k)}$ be given.
 \begin{enumerate}[(a)] 
 \item There are two vertex-disjoint monochromatic $\ell$-cycles in different colours together covering all but at most $4(k-\ell)$ vertices.
 \item
If $\ell \leq k/3$, the two  $\ell$-cycles cover all but at most $2(k-\ell)$ vertices.
	\end{enumerate}\end{thm}

Our proof does not use Bessy and Thomass\'e's theorem, nor does it rely on hypergraph regularity.
 
We suspect that a partition of all vertices into two cycles should always be possible. (It is not difficult to construct colourings which require at least two disjoint $\ell$-cycles for covering all the vertices, so this would be best possible.)

\begin{conjecture}
If $\ell, k, n\in \mathbb N$ with $n \equiv 0 \pmod{k-\ell}$, then every $2$-edge-colouring of $\K_n^{(k)}$ contains two vertex-disjoint monochromatic $\ell$-cycles in different colours covering all vertices. 
\end{conjecture}

An easy argument shows that 
for $\ell = k/2$ the conjecture is true. In order to see this, take any partition $\P$ of the vertices of $\K_n^{(k)}$ into sets $S_i, i \in [2n/k]$, of size $k/2$. Consider an auxiliary 2-edge-colouring of the complete graph on $\P$, giving $\{S_i,S_j\}$  the colour of  $S_i \cup S_j$ in $\K_n^{(k)}$. Bessy and Thomass\'e's theorem~\cite{BT10} yields two graph cycles, which correspond to  two disjoint  monochromatic $\ell$-cycles in different colours in~$\K_n^{(k)}$.

Also, we can obtain the following corollary from Theorem~\ref{thm:RB-ell-cycles}.
\begin{cor}
 Let $\ell, k, n\in \mathbb N$ such that
 $0<\ell \leq k/2$ and $k-\ell$ divides $n$. Then for any 2-edge-colouring of~$\K_n^{(k)}$, one can cover all the vertices of $\K_n^{(k)}$ with four vertex-disjoint monochromatic $\ell$-cycles, and if $\ell\leq k/3$, it can be done with three cycles instead of four. 
 \end{cor}
 This follows directly from our main theorem together with the observation that  the Ramsey number\footnote{The two colour Ramsey number of a $k$-uniform hypergraph $\h$ is the least integer $R(\h)$ for which every blue-red colouring of $\K_{R(\h)}^{(k)}$ contains a monochromatic copy of $\h$.}  of the $k$-uniform $\ell$-cycle of length two is $2(k-\ell)$. This can be seen by observing that  any $2$-edge-colouring of  $\K := \K_{2(k-\ell)}^{(k)}$ naturally defines a $2$-edge-colouring of $\K^* := \K_{2(k-\ell)}^{(k-2\ell)}$  by giving any edge $e^*$ in $\K^*$ the colour of $V(\K)\setminus e^*$ in $\K$. Then a monochromatic matching of size two in $\K^*$ corresponds to a monochromatic $\ell$-cycle of length two in $\K$.  Now, results of Alon, Frankl and Lov\'asz~\cite{AFL86} imply that the Ramsey number of a $2$-edge matching of uniformity $r$ is at most $2r+1$, which, since $2(k-2\ell)+1< 2(k-\ell)$, is enough for our purposes. 

\section{Partition into a path and a cycle}

We will identify a hypergraph $\h$ with its edge set, so when we write $e \in \h$ it refers to the edge $e$ of $\h$. Let us go through some necessary notation.

For an $\ell$-path or $\ell$-cycle $\X$, we order the $k$ vertices of each edge in such a way that the last $\ell$ vertices of an edge $e_i$ are the first $\ell$ vertices of the edge $e_{i+1}$.
For an edge $e = \{v_1,\dots,v_k\}$, we write $V_I(e)$ to denote the vertex set $\{v_i \in e: i \in I\}$, let $e^{-}$  denote the vertex set $V_{[\ell]}(e)$, let $e^{+}$ denote the vertex set $V_{[k]\setminus[k-\ell]}(e)$, and use $\mathring{e}$ for the set $e \setminus (e^{-} \cup e^{+})$.

An  $\ell$-path $\P$ of length $m$ is a {\em blue-red $\ell$-path} if there is $m_0 \in [m]$, called {\em turning point}, such that the $\ell$-paths $\{e_i : i \in [m_0]\}$ and $\{ e_i : i \in [m] \setminus [m_0] \}$ are monochromatic and have different colours. 

One of the first results in the field of monochromatic partitions was Gerensc\'er and Gy\'arf\'as' observation \cite{GG67} that every two-edge-coloured complete graph has a spanning blue-red path. We extend this observation to $\ell$-paths in hypergraphs in the following lemma.

\begin{lem}\label{lem:RBpath}
 Let $\ell, k, n\in\mathbb N$ such that $0<\ell \leq k/2$ and $k-\ell$ divides~$n$.
 Then every 2-edge-colouring of $\K_n^{(k)}$ contains a blue-red $\ell$-path $\mathcal P$ with $|V(\mathcal P)|= n- k+2\ell$.
\end{lem}

\begin{proof}
 Take a longest blue-red $\ell$-path $\P$ in $\K_n^{(k)}$, with edges $e_i$ for $i \in [m]$ and turning point $m_0$.
 Assume that all $e_i$ with $ i \in [m_0]$ are blue and all later edges on $\P$ are  red.

 If the set $Z$ of all vertices not covered by $\P$ has size $k-2\ell$, we are done.
 So assume otherwise; then $Z$ contains at least $2k-3\ell$ elements.
 Fix three disjoint sets $Z_0, Z_1, Z_2 \subset Z$ with $|Z_0| = \ell$ and $|Z_1| = |Z_2| = k-2\ell$.
 Since~$\P$ is maximal, we know that $e_R = e_1^{-} \cup Z_1 \cup Z_0$ is red, $e_B = e_{m}^{+} \cup Z_1 \cup Z_0$ is  blue, and $m_0 \neq m$.
 
 By colour symmetry, we can assume the edge $e = e_{m_0}^+ \cup Z_2 \cup Z_0$ is red. Then $(\P \setminus \{e_{m_0}\}) \cup \{e,e_R\}$ is a blue-red $\ell$-path longer than $\P$, which contradicts the maximality of $\P$.
\end{proof}

Observe that the blue-red $\ell$-path $\P$ given by Lemma~\ref{lem:RBpath} is as large as possible, since $|V(\P)| \equiv \ell \pmod{k-\ell}$ and $k-\ell$ divides~$n$.

Now we can show that there are a monochromatic $\ell$-path and a monochromatic $\ell$-cycle that together cover almost all the vertices.

\begin{lem}\label{lem:cycle+path}
 Let $\ell, k, n\in\mathbb N$  such that $0<\ell \leq k/2$ and $n = n_0(k-\ell)$ for some integer $n_0 \geq 3$. Then every 2-edge-colouring of $\K_n^{(k)}$ contains an $\ell$-cycle~$\C$  and an $\ell$-path $\P$  with the following properties:
 \begin{enumerate}
  \item $\C$ and $\P$  are vertex-disjoint;
  \item $\C$ and $\P$ are each monochromatic but use distinct colours;
  \item $\C$ has at least two edges; 
  \item if $\P\neq\emptyset$, then $|V(\mathcal C)|+|V(\mathcal P)|\in\{n-k+2\ell, n-2k+3\ell\}$; and
  \item if $\P=\emptyset$, then $|V(\mathcal C)|=n-k+\ell$.
 \end{enumerate}
\end{lem}

\begin{proof}
 By Lemma \ref{lem:RBpath}, there is a blue-red $\ell$-path  with edges $e_i$ for $i \in [m]$ and turning point $m_0$ that covers all but a set $Z$ of $k-2\ell$ vertices of~$\K_n^{(k)}$. 
 Among such $\ell$-paths, choose $\P_{\max}$ such that $\max\{m_0, m-m_0\}$ is as large as possible (i.e.~$\P_{\max}$ maximises the length of a monochromatic sub-$\ell$-path).
 By symmetry, we can assume that $\max\{m_0, m-m_0\}=m_0$, that $\P_B = \{e_i : i \in [m_0]\}$ is blue and that $\P_R = \{ e_i : i \in [m] \setminus [m_0] \}$ is  red. 
 Since $n\geq 3(k-\ell)$, we know that $m\geq 2$.
 If $m_0 < m$ and the edge $$e := e_{m_0}^{+} \cup  Z \cup e_{m}^{+}$$ is blue, then $\big(\P_{\max} \setminus e_{m_0+1}\big) \cup \{e\}$ is a blue-red $\ell$-path contradicting the choice of $\P_{\max}$. If $m_0 < m$ and the edge $e$ is red, then the $\ell$-cycle $\C_R = \P_R \cup e$ together with the $\ell$-path $\P_B \setminus \{e_{m_0}\}$ are as desired.

 So we can assume that $m_0=m$, that is, $\mathcal P_{\max}$ is all blue. If also one of the two edges $e_1^+\cup \mathring{e_1}\cup e_m^+$, $e_1^+\cup Z\cup e_m^+$ is blue, we can close $\P_{\max}$, forming an $\ell$-cycle that covers all but $e_1^-\cup Z$, or all but $e_1^-\cup\mathring{e_1}$, respectively, which is as desired. So we can suppose both edges $e_1^+\cup \mathring{e_1}\cup e_m^+$, $e_1^+\cup Z\cup e_m^+$ are red. They form an $\ell$-cycle with two edges, which together with $\P_{\max}\setminus \{e_1, e_2,  e_m\}$ covers all but $e_1^-\cup\mathring{e_2}\cup\mathring{e_m}$ (note that possibly $m=2$, in which case $e_2$  coincides with $e_m$).
 So, we found an $\ell$-cycle and an $\ell$-path which are as required (in particular, either they cover $n-2k+3\ell$  vertices, or the $\ell$-path is empty and the $\ell$-cycle covers  $n-k+\ell$ vertices).
\end{proof}

\section{Proof of Theorem~\ref{thm:RB-ell-cycles} (a)}\label{sec:k/2}

This section is devoted to the proof of Theorem~\ref{thm:RB-ell-cycles} (a).

Consider a monochromatic $\ell$-cycle $\C_B$ with at least two edges and a disjoint monochromatic $\ell$-path $\P_R$ as given by Lemma~\ref{lem:cycle+path}.
Note that if $\P_R$ has at most two edges, we are done, so assume otherwise. By deleting at most two edges from $\P_R$, if necessary, we can assume that \begin{equation*}\label{cover}
|V(\C_B \cup \P_R)|=n-3k+4\ell.
\end{equation*}

Among all such choices for $\C_B$ and $\P_R$ (including those where $\P_R$ is empty), assume we chose $\C_B$ and $\P_R$ such that

\begin{equation}\label{as:maxlength}
 \text{$\C_B$ has as many edges as possible.}
\end{equation}

By symmetry, we may assume that $\C_B$ is blue and $\P_R$ is red.
Say $\C_B$ has  edges $e_i, i \in [m_c]$ (and thus length $m_c\geq 2$), while $\P_R$  has edges $f_j, j \in [m_p]$ (and thus length $m_p\geq 0$).

Assuming that Theorem~\ref{thm:RB-ell-cycles} (a) does not hold, we will reach a contradiction by analysing the connections from the first/last edge of $\P_R$ to $\C_B$. If these cannot be used to close up $\P_R$ to an $\ell$-cycle, we find a red $\ell$-cycle on the same vertices as $\C_B$. In a last step, we will use this new red $\ell$-cycle together with $\P_R$ to form one large red $\ell$-cycle.

\smallskip 

We start by making a couple of easy observations. First of all, note that
\begin{equation}\label{mpgeq3}
 m_p \geq 2,
\end{equation}
 as otherwise $\C_B$ covers all but at most $4k-4\ell$ vertices.

Let $Z_1,Z_2,Z_3$ be mutually disjoint subsets of vertices not covered by $\C_B \cup \P_R$ such that $|Z_1| = |Z_2| = |Z_3| =  k-2\ell$. Consider the edges $$w^t := f_1^{-} \cup Z_t \cup f_{m_p}^{+},$$ for $t=1,2,3$. If any of the edges $w^t$ is red, then $\C_B$ together with $\P_R \cup \{w^t\}$ are $\ell$-cycles as in the theorem, covering all but $2k-2\ell$ vertices. So, 

\begin{equation}\label{as:w1w2}
 \text{$w^t := f_1^{-} \cup Z_t \cup f_{m_p}^{+}$ \ is blue, for $t=1,2,3$.}
\end{equation}

Consider the edges 
$$v^t_i := f_{m_p}^{+} \cup Z_t \cup e_i^-$$ for $i\in [m_c]$ and $t,t'\in\{1,2,3\}$. If for some triple $i,t,t'$ with $t\neq t'$, both  edges $v^t_i$, $v^{t'}_{i+1}$ are blue, then the $\ell$-cycle $\{v^t_i, v^{t'}_{i+1} \} \cup (\C_B \setminus \{e_i\})$ together with the $\ell$-path $\P_R \setminus \{f_{m_p}\}$ contradicts~\eqref{as:maxlength}. So,
for each $i\in [m_c]$, and each pair and $t,t'\in\{1,2,3\}$ with $t\neq t'$, 
\begin{equation}\label{vi}
 \text{one of the edges $v^t_i$, $v^{t'}_{i+1}$ is red.}
\end{equation}

Similary, for each $i\in [m_c]$, and each pair and $t,t'\in\{1,2,3\}$ with $t\neq t'$, setting $$u^t_i :=  f_1^{-} \cup Z_t \cup e_i^-,$$
we observe that
\begin{equation}\label{ui}
 \text{one of the edges $u^{t}_i, u^{t'}_{i+1}$ is red.}
\end{equation}

We now establish that our $\ell$-cycle is a bit longer than our $\ell$-path.

\begin{cla}\label{cla:longcycle}
 It holds that $m_c \geq m_p + 2$.
\end{cla}

\begin{proof}
 Suppose to the contrary that $m_c < m_p+2$. By~\eqref{ui}, we can assume the edge $u^1_1$ is red. 
 If the edge $v^2_1$ is red, too, then $\P_R \cup \{u^1_1, v^2_1\}$ and $\C_B \setminus \{e_1, e_{m_c}\}$  contradict the choice of $\C_B$ and $\P_R$ fulfilling~\eqref{as:maxlength}. So the edge  $v^2_1$ is blue, and thus by~\eqref{vi}, the edge $v_2^3$ is red.

 Now, if the edge $u^1_2$ is red, then $\P_R \cup \{v_2^3, u^1_2\}$ is a red $\ell$-cycle of length greater than $m_c$, which together with the path  $\C_B\setminus\{e_1, e_2\}$ contradicts~\eqref{as:maxlength}. Therefore,  $u^1_2$ is blue.
 But now, since by~\eqref{as:w1w2}, the edge $w^3$ is blue, we found a blue $\ell$-cycle, namely $\{v^2_1, w^3, u^1_2\} \cup (\C_B \setminus \{e_1\} )$, which together with the red $\ell$-path $\P_R \setminus \{f_1,f_{m_p} \} $ contradicts~\eqref{as:maxlength}. 
\end{proof}

Note that Claim~\ref{cla:longcycle} together with~\eqref{mpgeq3} implies that
\begin{equation}\label{mcgeq9}
m_c \geq 4.
\end{equation}

Let us now consider the edges $$g_i := (e_i \setminus e_i^+) \cup e_{i+2}^{+} \quad \mbox{and} \quad  h_i := (e_i \setminus e_i^+) \cup e_{i+3}^{+},$$ for $i \in [m_c]$ (considering all indices modulo $m_c$). The advantage of these edges is that on the one hand, each of these edges, if blue, provides a shortcut on $\C_B$ (and the vertices left out of $\C_B$ can be used for closing up $\P_R$). On the other hand, if all these edges are red, then they form new red $\ell$-cycles on the vertex set of $\C_B$.

 Let us  first show why any of the edges $g_i$, $h_i$ would be useful in blue.

\begin{cla}\label{cla:C_R}
The edges $g_i$  are red for all $i \in [m_c]$, and if $m_c>4$, then the edges $h_i$  are red for all $i \in [m_c]$.
\end{cla}

\begin{proof}
 Suppose that one of these edges $g_i$ or $h_i$ is blue (the latter only in the case that $m_c>4$).
 Then there is a blue cycle $\C_B'$ obtained from $\C_B$ by replacing the edges $e_i,e_{i+1},e_{i+2}$ with the edge $g_i$, or by replacing the edges $e_1,e_{i+1},e_{i+2}, e_{i+3}$ with the edge $h_i$.
 
 Consider the edges
 $u^1_{i+1}$ and $v^2_{i+1}$.
If both of these edges are red then the theorem holds, since $\C_B'$ together with $\P_R \cup \{u^1_{i+1}, v^2_{i+1}\}$ either covers all but $3k-3\ell$ vertices (if $g_i$ is blue); or cover all but $4k- 4\ell$ vertices (if $h_i$ is blue). So by symmetry, we can  assume that  $u^1_{i+1}$ is blue. Similarly, if the edges
$u^2_{i+2}$ and $v^3_{i+2}$ are both red then the theorem holds, so at least one of  them is blue.

Since $u^1_{i+1}$ is blue,~\eqref{vi} implies that $u^2_{i+2}$ is red, and thus $v^3_{i+2}$ is blue. Recall that by~\eqref{as:w1w2}, the edge $w^2$ is blue, too, and so,  the $\ell$-cycle $\{u^1_{i+1}, w^2, v^3_{i+2}\} \cup (\C_B \setminus \{e_{i+1}\})$ together with the $\ell$-path $\P_R \setminus \{f_1,f_{m_p}\}$ contradicts~\eqref{as:maxlength}.  
\end{proof}

Finally, consider the edge sets $$R_j := \{g_i: i \equiv j \pmod{3} \},$$ for $j = 0,1,2$. Notice that $R_0, R_1, R_2$ are three $\ell$-cycles of length $\frac{m_c}3$ if $m_c\equiv 0 \pmod{3}$, and together form one $\ell$-cycle otherwise. 

The remainder of the proof is split into several cases, depending on the value of $m_c$, and which of the edges $u^t_i$, $v^t_i$ are red. Note that by~\eqref{ui} (and after possibly renaming the edges on $\C_B$, or the sets $Z_t$), we may assume that $u_1^1$ is red.
Moreover, by~\eqref{vi}, at least one of the edges  $v_4^2$, $v_5^3$ is red (if $m_c=4$, we take indices modulo $4$, meaning that one of $v_4^2$, $v_1^3$ is red).

In all the cases considered below we use that $m_c\geq 4$ by~\eqref{mcgeq9}.

\begin{case}
	$m_c \not\equiv 0 \pmod{3}$ and $v_4^2$ is red.
\end{case}
In this case, consider the red $\ell$-cycle formed by $R_0\cup R_1\cup R_2$. We can substitute the edge $g_1$ from this $\ell$-cycle with the path $\{u^1_1\}\cup\P_R\cup \{v_4^2\}$ to obtain one red $\ell$-cycle which covers all but $2(k-\ell)$ vertices.

\begin{case}
	$m_c \equiv 0 \pmod{3}$ and $v_4^2$ is red.
\end{case}
Consider the auxiliary red $\ell$-cycle formed by
\begin{align*}
\{h_4\} \cup (R_2 \setminus \{g_2,g_5\}) &\cup \{h_2\} \cup (R_0 \setminus \{g_3\})\\
	&\cup \{h_3\} \cup (R_1 \setminus \{g_4\}).
\end{align*}
Similar as in the previous case, we can substitute the edge $g_1$ from the auxiliary $\ell$-cycle with the path $\{u^1_1\}\cup\P_R\cup \{v_4^2\}$ to obtain one red $\ell$-cycle which covers all but $3(k-\ell)$ vertices. (Note that this works fine even if $m_c=6$.)

\begin{case}
	$m_c \not\equiv 1 \pmod{3}$ and $v_5^3$ is red.
\end{case}
Consider the auxiliary red $\ell$-cycle formed by
\begin{align*}
\{h_1\} \cup (R_2 \setminus \{g_2\}) &\cup \{h_2\} \cup (R_0 \setminus \{g_3\})\\
&\cup \{h_3\} \cup (R_1 \setminus \{g_1,g_4\})
\end{align*}
and substitute its edge $h_1$  with $\{u^1_1\}\cup\P_R\cup \{v_5^3\}$ to obtain one red $\ell$-cycle which covers all but $3(k-\ell)$ vertices. (Note that this works even if $m_c=5$.)

\begin{case}
	$m_c \equiv 1 \pmod{3}$, $m_c\neq 4$, and $v_5^3$ is red.
\end{case}
If $m_c$ is odd, then we can use the cycle spanned by all edges $h_i$ except $h_1$, and the path $\P_R$, together with edges $u^1_1$ and $v_5^3$. This $\ell$-cycle covers all but $2(k-\ell)$ vertices. Otherwise, since $m_c>4$ and $m_c \equiv 1 \pmod{3}$, we know that $m_c\geq 10$. So we can use a similar approach as above, using five edges $h_i$ instead of two. 
More precisely, the red $\ell$-path formed by
\begin{align*}
\{h_5\} \cup (R_0 \setminus \{g_3,g_6\}) &\cup \{h_2, h_6\} \cup (R_1 \setminus \{g_1, g_4, g_7\})\\
&\cup \{h_3,h_7\} \cup (R_2 \setminus \{g_2,g_5,g_8\})
\end{align*}
and $\{u^1_1\}\cup\P_R\cup \{v_5^3\}$ covers all but $4(k-\ell)$ vertices.

\begin{case}
	$m_c= 4$, and $v_1^3$ is red while $v_4^2$ is blue.
\end{case}
Then by~\eqref{vi}, the edge $v_3^1$ is red, and we can  close $\P_R$ using the edges $v^3_1$, $v_3^1$, $g_1$, $g_4$. We covered all but $3(k-\ell)$ vertices.

This finishes the proof of the theorem.

\section{Proof of Theorem~\ref{thm:RB-ell-cycles} (b)}

In this section we give a  proof of Theorem~\ref{thm:RB-ell-cycles} (b), the first part of which follows very much  the lines of the proof of Theorem~\ref{thm:RB-ell-cycles} (a), while the last part is a bit different. In order to avoid repetition, we only sketch the first part, but give all details for the last part. We remark that much of the work can be avoided if we are only interested in two $\ell$-cycles covering all but $3(k-\ell)$ vertices instead of the  output of Theorem~\ref{thm:RB-ell-cycles} (b).

For the first part of the proof,
the main difference is that now, we use
 Lemma~\ref{lem:cycle+path} to find a monochromatic $\ell$-cycle $\C_B$ of length $m_c\geq 2$ with edges $e_i, i \in [m_c]$, and a disjoint monochromatic $\ell$-path $\P_R$  of length $m_p\geq 0$ with edges $f_j, j \in [m_p]$ such that
\begin{equation}\label{cover_k/3_}
|V(\C_B \cup \P_R)|=n-2k+3\ell,
\end{equation}
choosing $\C_B$ maximal under the these conditions.
That is, now we leave only $2(k-\ell)+\ell$ vertices uncovered in the beginning.
Instead of defining $Z_1, Z_2, Z_3$, we only define $Z_1, Z_2$  as two disjoint sets of vertices not covered by $\C_B \cup \P_R$ with $|Z_1| = |Z_2| =  k-2\ell$. The idea is that now, consecutive edges on $\C_B$ only intersect in at most $k/3$ vertices, which means that the interior of such an edge can be used in the same way as one of the sets $Z_t$. With this we can overcome the difficulty due to having only two sets $Z_t$ to operate with.

Again we easily show that $m_p \geq 2$ (using~\eqref{cover_k/3_}), and that edges $w^t$ defined for $t=1,2$ and $i\in [m_c]$, have the same properties as in the proof of Theorem~\ref{thm:RB-ell-cycles} (a). Now we define $v^t_i$ and $u^t_i$ as in that proof for $t=1,2$, and set $$v^3_i:= f_{m_p}^{+} \cup (e^i\setminus  e_i^+)\text{\ and \ }u^3_i:= f_{1}^{-} \cup (e^i\setminus  e_i^+).$$
It is easy to see that for each $i\in [m_c]$, and each pair and $t,t'\in\{1,2,3\}$ with $t\neq t'$ and $t'\neq 3$, at least one of the edges $v^t_i$, $v^{t'}_{i+1}$ is red and at least one of the edges $u^t_i$, $u^{t'}_{i+1}$ is red.

For showing that $m_c \geq m_p + 2$ (and thus  $m_c \geq 4$), observe that in the proof of Claim~\ref{cla:longcycle} in the proof of Theorem~\ref{thm:RB-ell-cycles} (a), there is only one time where we need that all three sets $Z_t$ are present, and that is at the very end, when we form the blue $\ell$-cycle $\{v^2_1, w^3, u^1_2\}\cup (\C_B \setminus \{e_1\})$. Instead, we can use the $\ell$-cycle $\{v^3_1, w^2, u^1_2\} \cup (\C_B \setminus \{e_1\})$.

For the rest of the proof one might define the edges $g_i$, $h_i$ as in the proof of Theorem~\ref{thm:RB-ell-cycles} (a), show they are red, and then go through Cases 1-5.\footnote{To see this goes through we remark that 
first, near the end of the proof of Claim~\ref {cla:C_R} of that theorem we  used to occupy the set $Z_3$, by employing the edge $v_{i+2}^3$. With the new definition of this edge, this works here too.\\
Second, when going through Cases 1-5, we cannot use the edges $u_1^1$, $v^2_4$, $v^3_5$ as before. This problem is easily overcome by first finding out which of $v^2_4$, $v^1_5$ is red. Say this is $v^1_5$ (otherwise rename all edges). Now, if the edge $u_2^2$ is red, then we are in the same situation as in Case 1,  2 or  5 of the proof of Theorem~\ref{thm:RB-ell-cycles} (a) (with indices augmented by one).
Otherwise, the edge $u_2^2$ is blue, and thus the edge $u_1^3$ is red, as  otherwise we could augment $\C_B$ using these two edges, and destroying one edge of~$\P_R$. Now we are in a situation that is very similar to Cases 3 and~4 of Theorem~\ref{thm:RB-ell-cycles} (a). As in these cases, neither of the edges $g_1$, $h_1$ was used, we have no problem finding our red $\ell$-cycle using $u_1^3$ instead of $u_1^1$.
}
 However, for establishing that edges $h_i$ are red, we would have to content ourselves with the outcome of two $\ell$-cycles covering all but $3(k-\ell)$ vertices. We can do a slightly better than that by arguing as follows.

Consider the edges $$a := f_1^{-} \cup \mathring{e}_3 \cup  V_{[2\ell] \setminus [\ell]}(e_2)$$ and $$a' := V_{[2\ell]\setminus [\ell]}(f_1) \cup \mathring{e}_4 \cup  V_{[k -\ell] \setminus [k - 2\ell]}(e_5)$$ (note that these edges are symmetric with respect to $e_3 \cap e_4$, as Figure~\ref{fig:1} shows).

\begin{figure}[h]
	\centering
	\begin{tikzpicture}
	\node[inner sep=0pt] (hyp) at (0,-0.4)
	{\includegraphics[width=.6\textwidth]{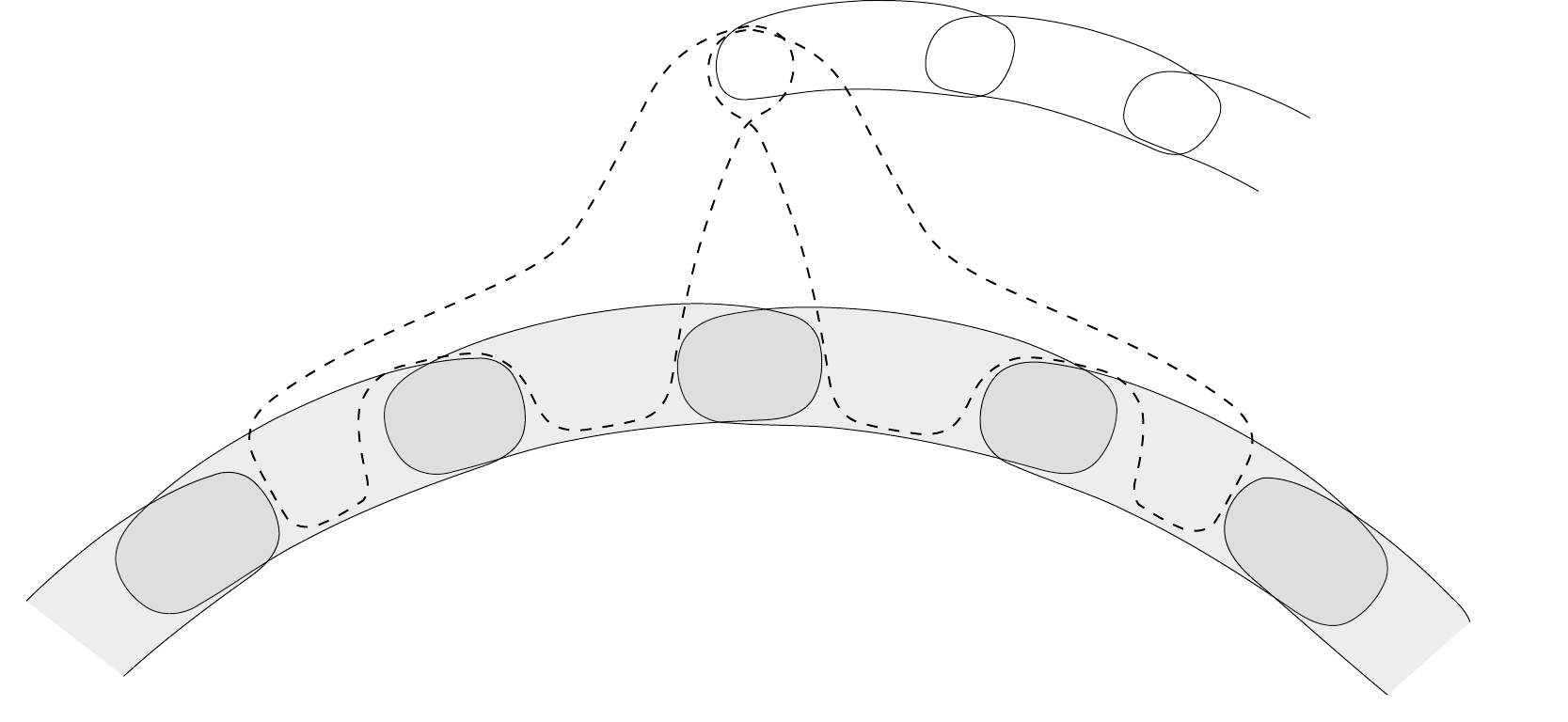}};
	\node (f1) at (0.5,1.1) {$f_1$};
	\node (dots-_p) at (2.8,0.5) {$\ddots$};
	\node (a) at (-0.65,0.15) {$a$};
	\node (a') at (0.35,0.2) {$a'$};
	\node[rotate=75] (dots-_c1) at (-3.8,-2.1) {$\ddots$};
	\node (e_2) at (-2,-1.5) {$e_2$};
	\node (e_3) at (-0.8,-1.1) {$e_3$};
	\node (e_4) at (0.5,-1.1) {$e_4$};
	\node (e_5) at (1.7,-1.5) {$e_5$};
	\node (dots-_c2) at (3.5,-2.1) {$\ddots$};
	\end{tikzpicture}
	\caption{Solid gray and solid white edges are blue and red edges, respectively.}
	\label{fig:1}
\end{figure}

If both $a$ and $a'$ are blue then we can replace the edges $e_3,e_4 \in \C_B$ with the edges $a, a', w^1$ and the blue\footnote{This edge is blue for the same reason for which $w^2$ is blue.} edge $V_{[2\ell]\setminus [\ell](f_1)}\cup Z_2\cup f_{m_p}^+$ to obtain a blue cycle which together with the red path $\P_R \setminus \{f_1,f_m\}$  contradicts the maximality of $\C_B$.
Therefore, we can assume that one of these edges is red, without loss of generality say
\begin{equation}\label{as:g1_k/3}
\text{$a$ is red.}
\end{equation}

 Next, consider the edges $$q_i:= V_{[2\ell] \setminus [\ell]}(e_i) \cup \mathring{e}_{i+1} \cup e_{i+3}^{+},$$ for $i \in [m_c]$. It is easy to see that  the edges $q_i$ form an $\ell$-cycle, which we will call $\C_R$.

\begin{cla}\label{cla:C_R_k/3}
	We may assume that $q_i$ is red, for all $i \in [m_c]$.
\end{cla}

\begin{proof}
	Suppose one of these edges, say $q_1$, is blue.
	Obtain $\C_B'$ from $\C_B$ by replacing the edges $e_2,e_3,e_4$ with the edge $q_1$. 
	
	First assume $u_3^3$ is blue. Then $u^1_4$ is red, by our analogue of~\eqref{ui}.
	Also,~$v^2_4$ is red, as otherwise we can replace $e_3$ with the edges $u_3^3$, $w^1$ and $v^2_4$, and thus contradicting the maximality of $\C_B$.
	But now, $\P_R \cup \{u^1_4,v^2_4\}$ is a red $\ell$-cycle, which, together with the blue $\ell$-cycle $\C_B'$ is as desired for the theorem.
	
	So from now on, assume that $u_3^3$ is red.
	Then, the edge $v^1_3$ is blue or we found $\ell$-cycles $\P_R \cup \{u_3^3,v^1_3\}$ and $\C_B'$ which are as desired for the theorem.	Now consider a set $Z'_2\subseteq Z_2$ of size $k-3\ell$. By the maximality of~$\C_B$ and taking  into account that $v^1_3$ is blue, we see that the edge $$b := f_{m_p}^{+} \cup Z'_2 \cup V_{[2\ell] \setminus [\ell]}(e_3) \cup e_3^+$$  has to be red (see Figure~\ref{fig:2}).
	
	\begin{figure}[h]
		\centering
		\begin{tikzpicture}
		\node[inner sep=0pt] (hyp) at (0,-0.4)
		{\includegraphics[width=.6\textwidth]{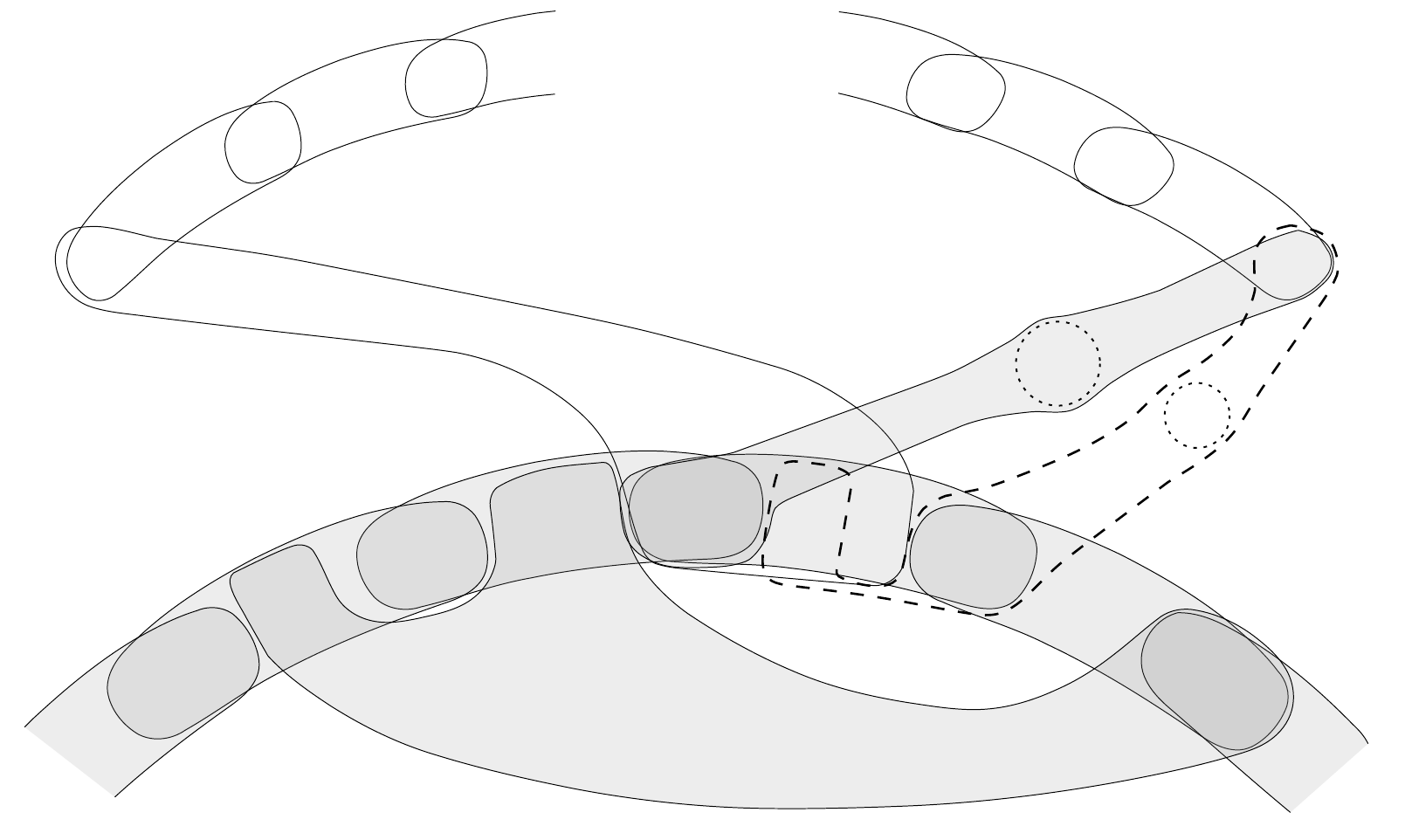}};
		\node (f1) at (-2.8,0.85) {$f_1$};
		\node (fmp) at (2.8,0.8) {$f_{m_p}$};
		\node (dots-_p) at (0,1.6) {$\dots$};
		\node (q1) at (0,-2) {$q_1$};
		\node (u33) at (-0.5,-0.2) {$u^3_3$};
		\node[rotate=75] (dots-_c1) at (-3.8,-2.4) {$\ddots$};
		\node (v13) at (1.3,0.2) {$v^1_3$};
		\node (b) at (2.6,-0.9) {$b$};
		\node (dots-_c2) at (3.65,-2.4) {$\ddots$};
		\end{tikzpicture}
		\caption{Diagram of edges $q_1, u_3^3, v_3^1$ and $b$. The dotted circles inside $v^1_3$ and $b$ are the sets $Z_1$ and $Z_2'$, respectively.}
		\label{fig:2}
	\end{figure}
	
	But then the $\ell$-cycles $\P_R \cup \{u_3^3,b\}$ and $\C_B'$ give the desired output  of the theorem. 
\end{proof}

We are now ready to prove Theorem~\ref{thm:RB-ell-cycles} (b). For this, first assume that $v^1_4$ is  red. Then, by~\eqref{as:g1_k/3} and by Claim~\ref{cla:C_R_k/3}, we know that $ (\C_R \setminus \{q_1,q_2\} ) \cup \P_R \cup \{a,v^1_4\}$ is a red $\ell$-cycle,  as desired for the theorem.

From now on assume $v^1_4$ is blue.
Then
$c := f_{m_p}^{+} \cup \mathring{e}_4 \cup V_{[2\ell]\setminus[\ell]}(e_5)$
is red, as otherwise the cycle obtained by replacing $e_4 \in \C_B$ with the edges $v^1_4, c$  yields a contradiction to the maximality of~$\C_B$. 
So, by \eqref{as:g1_k/3}, and since we chose $c$ so that it meets $q_5$ in exactly $\ell$ vertices, $$\C_R':=\big(\C_R \setminus \{q_2,q_3,q_4\}\big) \cup \P_R \cup \{a, c\}$$ is a red $\ell$-cycle covering all  vertices, except the $3(k-\ell)$ vertices lying in $$Z_1\cup Z_2\cup W \cup ( \mathring{e}_5\setminus c ) \cup e_5^+ \cup e_6^+ \cup e_7^+,$$ where $W$ is a set of $\ell$ vertices outside $\C_B\cup \P_R$ disjoint from $Z_1\cup Z_2$. 

Consider $d:= V_{[k-\ell]\setminus[k-2\ell]}(f_{m_p}) \cup Z_2\cup e_7^+$. Observe that either the $\ell$-cycle
$$\big(\C_R \setminus \{q_2,q_3\}\big) \cup \P_R \cup \{a, d\}$$ is red, and then it covers all but $2(k-\ell)$ vertices, as desired for the theorem, or the edge $d$ is blue, which we will assume from now on. Then 
by the maximality of $\C_B$,  the edge $d':=V_{[k-\ell]\setminus[k-2\ell]}(f_{m_p}) \cup Z_1\cup e_6^+$ is red. Now consider the edge $$e:=e_6^+ \cup Z_2 \cup e_7^+.$$ 
 If $e$ is blue, consider the blue path $e$ and the red cycle $\C_R'$ to obtain a contradiction to the maximality of $\C_B$ in the choice of $\C_B$ and $\P_R$.

So $e$ is red. Then
the $\ell$-cycle
$$\big(\C_R \setminus \{q_2,q_3\}\big) \cup \P_R \cup \{a, d', e\}$$ is red and covers all but $k-\ell$ vertices, as desired for the theorem.
 
This concludes the proof of Theorem~\ref{thm:RB-ell-cycles} (b).

\bibliographystyle{alpha}
\bibliography{BS17}

\end{document}